\documentclass[reqno,a4paper,12pt]{amsart}
\usepackage[english]{babel}
\usepackage{url}
\usepackage[%backref=page,
colorlinks=true, linkcolor = black, urlcolor=black, citecolor=blue, anchorcolor=blue]{hyperref}
\usepackage{comment}
\usepackage{tcolorbox}

\usepackage{cancel}

\newcommand{\crosses}[1]{%
	\ifcase#1\relax
	\or
	\rslash\or
	\rslash\mskip-5.5mu\rslash\or
	\rslash\mskip-5.5mu\rslash\mskip-5.5mu\rslash%
	\fi
}
\newcommand{\rslash}{\raisebox{.15ex}{/}}
\usepackage{xcolor,color}
\usepackage{geometry}
\usepackage[all, cmtip]{xy}
\usepackage{enumitem}
\usepackage{booktabs}
\usepackage{slashed}
\usepackage{bbm}
\usepackage{cancel}

\usepackage{cleveref}
\usepackage{tikz}
\usepackage{tikz-cd}
\tikzcdset{arrow style=tikz, diagrams={>=stealth}}

\usepackage{xargs}  % Use more than one optional parameter in a new commands
\usepackage{soul}

\usepackage{amsmath}
\usepackage{amssymb,graphicx}
\usepackage{amsthm}
\usepackage{latexsym}
\usepackage{amsfonts}
\usepackage{bbm}
\usepackage{cases}

\calclayout

\numberwithin{equation}{section}

\theoremstyle{plain}
\newtheorem{lemma}{Lemma}[section]
\newtheorem{proposition}[lemma]{Proposition}
\newtheorem{proposition/definition}[lemma]{Proposition/Definition}
\newtheorem{theorem}[lemma]{Theorem}

\newtheorem*{theorem*}{Theorem}

\theoremstyle{definition}
\newtheorem{remark}[lemma]{Remark}
\newtheorem{example}[lemma]{Example}

\DeclareRobustCommand{\SkipTocEntry}[5]{}

\allowdisplaybreaks

\title{Integrating Nijenhuis Structures}

\makeatletter
\def\author@andify{%
	\nxandlist {\unskip ,\penalty-1 \space\ignorespaces}%
	{\unskip {} \@@and~}%
	{\unskip \penalty-2 \space \@@and~}%
}
\makeatother

\author[F.~Pugliese]{Fabrizio Pugliese}
\address{DipMat, Universit\`a degli Studi di Salerno, via Giovanni Paolo II n${}^{\circ}$123, 84084 Fisciano (SA), Italy.}
\email{\href{mailto:fpugliese@unisa.it}{fpugliese@unisa.it}}

\author[G.~Sparano]{Giovanni Sparano}
\address{DipMat, Universit\`a degli Studi di Salerno, via Giovanni Paolo II n${}^{\circ}$123, 84084 Fisciano (SA), Italy.}
\email{\href{mailto:sparano@unisa.it}{sparano@unisa.it}}

\author[L.~Vitagliano]{Luca Vitagliano}
\address{DipMat, Universit\`a degli Studi di Salerno, via Giovanni Paolo II n${}^{\circ}$123, 84084 Fisciano (SA), Italy.}
\email{\href{mailto:lvitagliano@unisa.it}{lvitagliano@unisa.it}}

\keywords{Lie groupoids, Lie algebroids, Nijenhuis structures, Multiplicative tensors}

\subjclass[2010]{22A22 (Primary), %Lie groupoids
53D17, %Poisson manifolds, groupoids and algebroids
58A50%graded manifolds
}

\begin{document}

\begin{abstract}
A Nijenhuis operator on a manifold $M$ is a $(1,1)$ tensor $\mathcal N$ whose Nijenhuis-torsion vanishes. A Nijenhuis operator $\mathcal N$ on $M$ determines a Lie algebroid structure $(TM)_{\mathcal N}$ on the tangent bundle $TM$. In this sense a Nijenhuis operator can be seen as an infinitesimal object. In this paper, we identify its \emph{global counterpart}. Namely, we show that when the Lie algebroid $(TM)_{\mathcal N}$ is integrable, then it integrates to a Lie groupoid equipped with appropriate additional structure responsible for $\mathcal N$, and viceversa, the Lie algebroid of a Lie groupoid equipped with such additional structure is of the type $(TM)_{\mathcal N}$ for some Nijenhuis operator $\mathcal N$. We illustrate our integration result in various examples.
\end{abstract}

%\begin{abstract}
%\end{abstract}
%\date{\today}
\maketitle

\tableofcontents

\section{Introduction}

There is a natural \emph{integrability condition} that one can impose on a $(1,1)$ tensor $\mathcal N$ on a manifold $M$. Namely, $\mathcal N$ defines a skew-symmetric $(2, 1)$ tensor $T_{\mathcal N}$, its \emph{Nijenhuis torsion}, which is given by
\[
T_{\mathcal N} (X, Y) = [\mathcal NX, \mathcal NY] + \mathcal N^2[X,Y] - \mathcal N[\mathcal NX, Y] - \mathcal N [X, \mathcal NY], \quad X, Y \in \mathfrak X (M).
\]
Equivalently $T_{\mathcal N} = \frac{1}{2}[\mathcal N, \mathcal N]^{\mathrm{fn}}$, where $[-,-]^{\mathrm{fn}}$ is the Fr\"olicher-Nijenhuis bracket of vector valued forms. We say that $\mathcal N$ is \emph{integrable} when $T_{\mathcal N} = 0$ identically, in this case we call $\mathcal N$ a \emph{Nijenhuis operator}. Nijenhuis operators naturally appear in relation to other geometries: for instance in complex geometry as integrable almost complex structures, and in integrable systems as recursion operators for bi-Hamiltonian systems (see, e.g., the review \cite{KS2017} and references therein). But they also appear elsewhere (see, e.g., references in \cite{BKM2022} for a partial list). Recently, Bolsinov, Konyaev and Matveev, in a series of papers \cite{BKM2022, K2019, BKM2020, BKM2021, BKM2021b} initiated a project consisting in systematically studying Nijenhuis operators in their own. For instance, in the first paper of the series \cite{BKM2022} they discuss local normal forms, in the same spirit as Weinstein splitting theorem for Poisson structures \cite{W1983}, while in the second paper of the series \cite{K2019} Konyaev discusses the \emph{linearization problem}. 

In fact every Nijenhuis operator defines a Lie algebroid structure $(TM)_{\mathcal N}$ on the tangent bundle $TM \to M$. The anchor of $(TM)_{\mathcal N}$ is $\mathcal N$ itself, while the Lie bracket on sections is given by
\[
[X, Y]_{\mathcal N} = [\mathcal N X, Y] +[X, \mathcal N Y] - \mathcal N[X,Y], \quad X, Y \in \mathfrak X (M).
\]

Recall that a Lie algebroid is the \emph{infinitesimal counterpart} of a \emph{global object}: a Lie groupoid. When the Lie algebroid $(TM)_{\mathcal N}$ integrates to a Lie groupoid $G \rightrightarrows M$, it is natural to wonder what kind of structure on $G$ is responsible for the Nijenhuis operator defining $(TM)_{\mathcal N}$. Similar questions were already posed and answered in various analogous situations. For instance, a Poisson tensor $\pi$ on a manifold $M$ defines a Lie algebroid structure $(T^\ast M)_\pi$ on the cotangent bundle. When the Lie algebroid $(T^\ast M)_\pi$ integrates to a Lie groupoid, there is a \emph{multiplicative symplectic structure} on the source simply connected integration which is responsible for $\pi$. In this sense, a \emph{symplectic groupoid}, i.e.~a Lie groupoid equipped with a multiplicative symplectic structure, is the global counterpart of a Poisson manifold \cite{K1987,W1987}. Similarly, a holomorphic groupoid is the global counterpart of a holomorphic Lie algebroid \cite{LSX2009}, a contact groupoid is the global counterpart of a Jacobi manifold \cite{KS1993,L1993,CS2015}, etc. (see the review \cite{KS2016} for more examples).

In this paper we characterize Lie groupoids $G \rightrightarrows M$ integrating a Nijenhuis operator, i.e.~Lie groupoids whose Lie algebroid is isomorphic to $(TM)_{\mathcal N}$ for some Nijenhuis operator $\mathcal N$ on $M$. In particular we identify what precise structure on $G$ is responsible for $\mathcal N$. In order to guess the final answer we begin with a closer look at the infinitesimal picture. Recall that the datum of a Lie algebroid can be encoded in that of a differential graded (DG) manifold concentrated in degrees $0,1$. The DG manifold encoding the Lie algebroid $A \Rightarrow M$ is $(A[1], d_A)$ where $A[1]$ is obtained from $A$ by shifting by one the degree of the fiber coordinates and $d_A$ is the Lie algebroid De Rham differential. Our first result is

\begin{theorem}\label{theor:1}
The datum of a Lie algebroid of the type $(TM)_{\mathcal N}$ for some Nijenhuis operator $\mathcal N$ on $M$ is equivalent to that of a DG manifold $(A[1], d_A)$ concentrated in degrees $0,1$ equipped with an integrable almost tangent structure $V$ of internal degree $-1$ such that
$
[[d_A, V]^{\mathrm{fn}}, V]^{\mathrm{fn}} = 0.
$
\end{theorem}

Our main result is the following global version Theorem \ref{theor:1}.

\begin{theorem}\label{theor:2}
A Lie groupoid $G \rightrightarrows M$ with source $s : G \to M$, target $t : G \to M$ and Lie algebroid $A \Rightarrow M$ integrates a Nijenhuis operator $\mathcal N$ on $M$ if and only if it is equipped with a vector bundle map $U : TM \to A$ such that 1) $\ker \overrightarrow U = \ker dt$, 2) $\operatorname{im} \overrightarrow U = \ker ds$, and 3) $[\overrightarrow U, \overrightarrow U]^{\mathrm{fn}} = 0$, i.e.~$\overrightarrow U$ is a Nijenhuis operator.
\end{theorem}

In the previous statement, $\overrightarrow U$ is a $(1,1)$ tensor on $G$ defined as an appropriate \emph{right invariant lift} of $U$. There is an analogy between Theorems \ref{theor:1} and \ref{theor:2} which we now explain. First of all, an almost tangent structure of degree $-1$ on $A[1]$ as in Theorem is equivalent to a vector bundle isomorphism $U : TM \to A$. Similarly, conditions 1) and 2) in Theorem \ref{theor:2} together are equivalent to $U$ being bijective. Secondly, condition 3) in Theorem \ref{theor:2} is actually equivalent to $[\delta U, \overrightarrow U]^{\mathrm{fn}}$ where $\delta$ is a certain global analogue of the Lie derivative $[d_A, -]^{\mathrm{fn}}$ of a $(1, 1)$ tensor on $A[1]$ along the homological vector field.

We interpret the $(1,1)$ tensor $\overrightarrow U$ on $G$ of Theorem \ref{theor:2} as the structure responsible for the Nijenhuis operator defining the Lie algebroid of $G$. In this spirit, Theorem \ref{theor:2} is an \emph{integration result} for Nijenhuis operators. However, we stress that such result is slightly different in nature from the integration theorem of, e.g., a Poisson structure $\pi$. While, in general, we can guarantee the existence of a multiplicative symplectic form only on a source-simply connected integration of the Lie algebroid $(T^\ast M)_\pi$, a $(1,1)$ tensor $\overrightarrow U$ as in Theorem \ref{theor:2} exists on \emph{every} Lie groupoid integrating $(TM)_{\mathcal N}$.

The paper is organized as follows. In Section \ref{Sec:1} we recollect the necessary material about $(1,1)$ tensors on Lie groupoids, Lie algebroids and graded manifolds. This material is not novel except for Lemma \ref{lem:delta_U} that will play an important role in the rest of the paper. A version of the cochain complex $(C^\bullet_{\mathrm{def}}(G, T^{1,0}), \delta)$ of Section \ref{Sec:1} appeared already in \cite{PSV2021}. In Section \ref{Sec:2} we state and prove Theorem \ref{theor:1} (see Theorem \ref{theor:L_dV}). In Section \ref{Sec:3} we prove our main result Theorem \ref{theor:2} (see Theorem \ref{theor:delta_U}) and illustrate it with a few trivial examples. In Section \ref{Sec:4} we provide more illustrative examples, including a curious Lie groupoid structure on the double tangent bundle $TTB$ of a manifold $B$ which, to the best of our knowledge, is new (see Subsection \ref{Subsec:5}). We also discuss linear Nijenhuis operators on a vector space.

Finally, we want to mention the recent work of Bursztyn, Drummond, Netto \cite{D2020, BDN2021} on Nijenhuis operators in connections to Lie groupoids, Lie algebroids and related structure (see also \cite{D2019,DCG}). The present paper goes in a complementary direction. While those authors consider Nijenhuis operators on Lie groupoids (resp.~Lie algebroids) integrating (resp.~encoding) other geometric structures, we consider Lie groupoids integrating Nijenhuis operators themselves. 

\bigskip

\textbf{Notation.} We denote by $G \rightrightarrows M$ a Lie groupoid with $G$ its space of arrows, and $M$ its space of objects. We denote by $A \Rightarrow M$ a Lie algebroid with $A \to M$ its underlying vector bundle. We denote by $[-,-]_A$ the Lie bracket on sections of $A$ and by $\rho_A : A \to TM$ the anchor. Given a surjective submersion $\pi : M \to B$ we denote by $T^\pi M := \ker d\pi$ the vertical tangent bundle with respect to $\pi$.

We assume the reader is familiar with Lie groupoids and Lie algebroids. Our main reference for this material is the lecture notes of Crainic and Fernandes \cite{CF2011} (see also \cite{mackenzie}). Additionally, we assume some familiarity with graded geometry, including $DG$ manifolds, for which the reader might consult \cite{M2006}.

\section{A Review of Multiplicative and IM $(1,1)$ Tensors}\label{Sec:1}

In the rest of the paper we will extensively work with $(1,1)$ tensors on Lie groupoids, Lie algebroids and graded manifolds (of a certain type). In this section we summarize the necessary material. Our main references here are \cite{LSX2009, BD2013, BD2019}. We will often interpret $(1,1)$ tensors as vector valued $1$-forms. 

Let $G \rightrightarrows M$ be a Lie groupoid with Lie algebroid $A \Rightarrow M$. The structure maps of $G$ will be denoted $s, t : G \to M$ (source and target), $m : G^{(2)} \to M$ (multiplication), $u : M \to G$ (unit), and $i : G \to G$ (inversion). We denote by $G^{(k)}$ the manifold of $k$ composable arrows in $G$. Let $T \in \Omega^1 (G, TG)$. There are several equivalent ways of stating the compatibility of $T$ and the groupoid structure. The easiest one is the following: we say that $T$ is a \emph{multiplicative $(1,1)$ tensor} if $T : TG \to TG$ is a groupoid map with respect to the tangent groupoid structure $TG \rightrightarrows TM$, in particular, there exists a $(1,1)$ tensor $T^M$ on $M$ such that $T$ and $T^M$ are both $s$-related and $t$-related.

Multiplicative $(1,1)$ tensors on $G$ can also be seen as $0$-cocycles in an appropriate cochain complex $C{}^\bullet_{\mathrm{def}} (G, T^{1,0})$ that we now describe:

First of all, $C{}^\bullet_{\mathrm{def}} (G, T^{1,0})$ will be concentrated in degrees $k \geq -1$. Degree $-1$ cochains are vector bundle maps
\[
U : TM \to A.
\]
For $k \geq 0$, degree $k$ cochains are vector bundle maps
\[
U : T G^{(k + 1)} \to  TG, 
\]
covering the projection $\mathrm{pr}_1 : G^{(k + 1)} \to G$, $(g_1, \ldots, g_{k+1}) \mapsto g_1$ onto the first factor, for which there exists another vector bundle map
\[
U^M : T G^{(k)} \to TM,
\]
covering the projection $t \circ \mathrm{pr}_1 : G^{(k)} \to M$, such that the following diagram commutes: 
\[
\begin{array}{c}
\xymatrix{ T G^{(k+1)} \ar[r]^-{U} \ar[d]_-{d(\mathrm{pr}_2 \times \cdots \times \mathrm{pr}_{k+1})} &  TG  \ar[d]^-{ds} \\
T G^{(k)} \ar[r]^-{U^M} & TM}
\end{array},
\]
where $\mathrm{pr}_i : G^{(k+1)} \to G$ is the projection onto the $i$-th factor.
For instance, a $(1,1)$ tensor $T$ on $G$ belongs to $C^0_{\mathrm{def}}(G, T^{1,0})$ if and only if it is $s$-projectable, i.e.~it is $s$-related to some $(1,1)$ tensor on $M$.

Next we describe the differential $\delta : C{}^\bullet_{\mathrm{def}} (G, T^{1,0}) \to C{}^{\bullet + 1}_{\mathrm{def}} (G, T^{1,0})$.  We begin with its action on a degree $-1$ cochain  $U : T M \to A$. First define two $(1,1)$ tensors $\overrightarrow U$ and $\overleftarrow U$ on $G$, by putting
\begin{equation}\label{eq:arrow_U}
\overrightarrow U_g = dR_g \circ U_{t(g)} \circ dt  \quad \text{and} \quad \overleftarrow U_g = dL_g \circ di \circ U_{s(g)} \circ ds,
\end{equation}
for every $g \in G$, where $R_g$ (resp. $L_g$) denotes right (resp. left) translation by $g$. Now 
\[
\delta U := \overrightarrow U + \overleftarrow U
\]
which is a well defined $0$-cochain whose $M$-projection $(\delta U)^M$ is given by
\[
(\delta U)^M = \rho_A \circ U : TM \to TM,
\]
where $\rho_A : A \to TM$ is the anchor map.

Finally, let $U \in  C{}^{k}_{\mathrm{def}} (G, T^{1,0})$, with $k \geq 0$. For all $(v_1, \ldots, v_{k+2}) \in TG^{(k+2)}$ we put
\begin{equation}\label{eq:delta_bar}
\begin{aligned}
 \delta U (v_1, \ldots, v_{k+2}) 
& = - U \left( v_1 v_2, v_3, \ldots, v_{k+2}\right)  U \left( v_2, \ldots, v_{k+2}\right)^{-1} \\
& \quad + \sum_{i = 2}^{k+1} (-)^{i} U\left(v_1, \ldots, v_i  v_{i+1}, \ldots, v_{k+2} \right)  + (-)^{k} U (v_1, \ldots, v_{k+1} ),
\end{aligned}
\end{equation}
where we used the multiplication and inversion in the tangent Lie groupoid $TG \rightrightarrows TM$. Then $\delta U$ is a well-defined $(k+1)$-cochain whose $M$-projection $(\delta U)^M$ is given by
\[
\begin{aligned}
 (\delta U)^M (v_2, \ldots, v_{k+2})
& = -dt \big(U (v_2, \ldots, v_{k+2}) \big) \\
& \quad + \sum_{i = 2}^{k+1} (-)^{i+1} U^M (v_2, \ldots, v_iv_{i+1}, \ldots, v_{k+2})  + (-)^{k} U^M (v_2, \ldots, v_{k+1}).
\end{aligned}
\]
The operator $\delta$ is indeed a differential. The definition of $C{}^{\bullet}_{\mathrm{def}} (G, T^{1,0})$ is very similar to that of Crainic-Mestre-Struchiner deformation complex of a Lie groupoid $G \rightrightarrows M$ \cite{CMS2018}, and the definition of the differential $\delta$ is formally identical up to replacing points in $G^{(k+2)}$ with points in $TG^{(k+2)}$. This is the main reason why we adopt a similar notation $C{}^{\bullet}_{\mathrm{def}} (G, T^{1,0})$ for our complex. Notice that there are versions of this complex for higher order tensors (see, e.g.~\cite{PSV2021}). We speculate that cohomology classes of $C{}^{\bullet}_{\mathrm{def}} (G, T^{1,0})$ should be seen as shifted $(1,1)$ tensors on the differentiable stack $[G / M]$ represented by $G$, but we will not explore this point of view here.

Next we discuss infinitesimal multiplicative (IM) $(1,1)$ tensors on Lie algebroids. We begin with linear $(1,1)$ tensors on vector bundles. Let $A \to M$ be a vector bundle %over $M$, with projection $\pi : A \to M$
, and let $T : TA \to TA$ be a $(1,1)$ tensor on the total space $A$. We say that $T$ is \emph{linear} if it is multiplicative with respect to the Lie groupoid structure on $A$ given by fiber-wise addition, equivalently $T$ is a vector bundle map with respect to the vector bundle structure $TA \to TM$. This definition emphasizes the fiber-wise addition in $A$. There is an equivalent definition emphasizing the fiber-wise scalar multiplication which is often useful in practice. Namely, a $(1,1)$ tensor on $A$ is linear if and only if it is of degree $0$ with respect to the action $h : \mathbb R \times A \to A$ of the multiplicative monoid $\mathbb R$ on $A$ given by fiber-wise scalar multiplication, i.e.~$h_r^\ast T = T$ for all $r \in \mathbb R \smallsetminus 0$. In the following, we will also need $(1, 1)$-tensors $S \in \Omega^1 (A, TA)$ of degree $-1$ with respect to $h$, i.e.~$h_r^\ast S = r^{-1} S$ for all $r \in \mathbb R \smallsetminus 0$. We call such tensors \emph{core} $(1,1)$ tensors adopting a terminology that we already used in \cite{PSV2021}. Both core and linear $(1,1)$ tensors can be encoded into certain sections of appropriate vector bundles over $M$. To see this, first recall that degree $-1$ vector fields (also called core vector fields in \cite{PSV2021, PSV2021b}) on $A$ identify with sections $a$ of $A$ via the vertical lift $a \mapsto a^\uparrow$, and degree $0$ vector fields (also called linear vector fields) on $A$ identify with sections $D$ of the gauge algebroid $DA \Rightarrow M$ via the map $D \mapsto X_D$, where $X_D \in \mathfrak X (A)$ is the linear vector field uniquely determined by $[X_D, a^\uparrow] = (Da)^\uparrow$ for all $a \in \Gamma (A)$. Similarly, core $(1,1)$ tensors on $A$ identify with vector bundle maps $U : TM \to A$ via the map $U \mapsto U^\uparrow$, where $U^\uparrow \in \Omega^1 (A, TA)$ is the core $(1,1)$ tensor uniquely determined by $U^\uparrow (X_D) = U(\sigma_D)^\uparrow$, and $\sigma_D \in \mathfrak X (M)$ is the symbol of $D$. Finally, linear $(1,1)$ tensors on $A$ identify with triples $(\mathcal D, \ell, T^M)$ where $\ell : A \to A$ is a vector bundle map, $T^M \in \Omega^1 (M, TM)$ is a $(1,1)$-tensor on $M$, and $\mathcal D : \Gamma (A) \to \Omega^1 (M, A)$ is a differential operator such that
\[
\mathcal D (fa) = f \mathcal D(a) + df \otimes \ell (a) - \langle df, T^M \rangle \otimes a, \quad a \in \Gamma (A), \quad f \in C^\infty (M),
\]
via the map $(\mathcal D, \ell, T^M) \mapsto T$ where $T \in \Omega^1 (A, TA)$ is the linear $(1,1)$ tensor uniquely determined by its $M$-projection being $T^M$ and, additionally,
\[
\mathcal L_{a^\uparrow} T = \mathcal D(a)^\uparrow, \quad  T a^\uparrow= \ell (a)^\uparrow.
\]

When $A \Rightarrow M$ is a Lie algebroid with Lie bracket $[-,-]_A$ and anchor map $\rho_A$, then a $(1,1)$ tensor $T \in \Omega^1 (A, TA)$ is IM if $T : TA \to TA$ is a Lie algebroid map with respect to the tangent algebroid structure $TA \Rightarrow TM$. In particular $T$ is a linear $(1,1)$ tensor and $T$ being IM is equivalent to the associated triple $(\mathcal D, \ell, T^M)$ satisfying the following identities \cite{DE2019, BD2019}:
\[
\begin{aligned}
\mathcal D \big([a, b]_A\big) & = \mathcal L_a \mathcal D (b) - \mathcal L_b \mathcal D (a) \\
\ell \big( [a, b]_A\big) & = [a, \ell (b)] - \iota_{\rho_A (b)}\mathcal D (a), \\
\mathcal L_{\rho_A (a)} T^M & = \rho_A \circ \mathcal D (a), \\
T^M \circ \rho_A & = \rho_A \circ \ell,
\end{aligned}
\]
for all $a, b \in \Gamma (A)$.

If $T : TG \to TG$ is a multiplicative $(1,1)$ tensors on the Lie groupoid $G \rightrightarrows M$, then, by differentiation, we get an IM $(1,1)$ tensor $\dot T : TA \to TA$ on the Lie algebroid $A \Rightarrow M$ of $G$. The triple $(\mathcal D, \ell, T^M)$ corresponding to $\dot T$ is given by
\[
\mathcal D(a) = \mathcal L_{\overrightarrow a} T |_{TM}, \quad \ell (a) = T \overrightarrow a |_M,
\]
for all $a \in \Gamma (A)$, while $T^M$ is the $M$-projection of $T$ (here $\overrightarrow a \in \mathfrak X (G)$ is the right invaraint vector field corresponding to $a$).

IM $(1,1)$ tensors can also be seen as $0$-cocycles in a cochain complex. The easiest way to see this is via graded geometry. We assume the reader is familiar with graded manifolds and homological vector fields. We only recall that an $\mathbb N$-manifold of degree $k$ is a graded manifold whose function algebra is generated in degree $0, 1, \ldots, k$. An $\mathbb NQ$-manifold is an $\mathbb N$-manifold equipped with a $Q$-structure, i.e.~a (degree 1) homological vector field $Q$. We will also need to consider vector valued forms on $\mathbb N$-manifolds. Our conventions on differential forms on an $\mathbb N$-manifold are as follows. Given an $\mathbb N$-manifold $\mathcal M$, differential forms on $\mathcal M$ are fiber-wise polynomial functions on the shifted tangent bundle $T[1] \mathcal M$ and are denoted $\Omega^\bullet (\mathcal M)$. Similarly, vector valued forms on $\mathcal M$ are fiber-wise polynomial sections of the pull-back bundle $T[1] \mathcal M \times_{\mathcal M} T \mathcal M \to T[1] \mathcal M$ and are denoted $\Omega^\bullet (\mathcal M, T\mathcal M)$. Both $\Omega^\bullet (\mathcal M)$ and $\Omega^\bullet (\mathcal M, T\mathcal M)$ are bi-graded vector spaces, the two gradings being the form degree, and the internal, i.e.~coordinate, degree. The total degree of a form is then the sum of the internal and the form degrees, and the usual Koszul sign rule holds with respect to the total degree. For instance, the Fr\"olicher-Nijenhuis bracket on $\Omega^\bullet (\mathcal M, T\mathcal M)$ is a graded Lie bracket of total degree $0$. 

Remember that the datum of a vector bundle $A \to M$ is equivalent to the datum of an $\mathbb N$-manifold of degree $1$ via $A \leadsto A[1]$ (where $A[1]$ is the $\mathbb N$-manifold obtained by shifting by $1$ the fiber degree in $A$), and the datum of a Lie algebroid $A \Rightarrow M$ is equivalent to the datum of an $\mathbb N Q$-manifold of degree $1$ via $A \leadsto (A[1], d_A)$ where $d_A$ is the Lie algebroid De Rham differential. Exactly as for usual vector bundles, degree $-1$ vector fields on $A[1]$ identify with sections of $A$ and degree $0$ vector fields identify with sections of $DA$. Similarly, $(1,1)$ tensors on $A[1]$ of internal degree $-1$ identify with vector bundle maps $U : TM \to A$, hence with core $(1,1)$ tensors on $A$, and $(1,1)$ tensors on $A[1]$ of internal degree $0$ identify with triples $(\mathcal D, \ell, T^M)$ as above, hence with linear $(1,1)$ tensors on $A$. If $T \in \Omega^1 (A, TA)$ is a linear $(1,1)$ tensor, we denote by $T^{[1]} \in \Omega^1 (A[1], TA[1])$ the corresponding $(1,1)$ tensor (of internal degree $0$) on $A[1]$. If $A \Rightarrow M$ is a Lie algebroid, then $T$ is IM if and only if $\mathcal L_{d_A} T^{[1]} = 0$ \cite{PSV2021}, hence the assignment $T \mapsto T^{[1]}$ identifies IM $(1,1)$ tensors on $A$ with $0$-cocycles in the cochain complex $\big(\Omega^1 (A[1], TA[1]), \mathcal L_{d_A} \big)$ (whose grading is given by the internal degree). 

\begin{example}\label{ex:N^tan}
Let $T \in \Omega^1 (M, TM)$ be a $(1,1)$ tensor on $M$. Its \emph{tangent lift} is the linear $(1,1)$ tensor $T^{\mathrm{tan}} \in \Omega^1 (TM, TTM)$ on $TM$ corresponding to the triple
\[
\big([-, T]^{\mathrm{fn}}, T, T\big)
\]
(see, e.g., \cite{BDN2021}) where $[-, -]^{\mathrm{fn}}$ is the Fr\"olicher-Nijenhuis bracket of vector valued forms.
\end{example}

Finally, let $A \Rightarrow M$ be the Lie algebroid of a Lie groupoid $G \rightrightarrows M$, and let $U : TM \to A$ be a vector bundle map. Then $U$ can be seen both as a degree $-1$ conchain in the cochain complex $\big(C_{\mathrm{def}}^\bullet (G, T^{1,0}), \delta \big)$ and in the cochain complex $\big(\Omega^1 (A[1], TA[1]), \mathcal L_{d_A} \big)$ (via $U \mapsto U^\uparrow$). To the best of our knowledge the following lemma is new.

\begin{lemma}\label{lem:delta_U}
The triples $(\mathcal D, \ell, T^M)$ corresponding to $\dot{\delta U}$ and $\mathcal L_{d_A} U^\uparrow$ are both given by
\begin{equation}\label{eq:D,l,T,deltaU}
\mathcal D (a) = \mathcal L^A_a U, \quad 
\ell =  U \circ \rho_A, \quad 
T^M = \rho_A \circ U,
\end{equation}
for all $a \in \Gamma (A)$ (where $\mathcal L^A$ is the Lie algebroid Lie derivative, see below), hence
\[
 \dot{(\delta U)}{}^{[1]} = \mathcal L_{d_A} U^\uparrow.
\]
\end{lemma}

\begin{proof}
Begin with $\dot{\delta U}$. Let $a \in \Gamma (A)$ and notice that the vector field $\overrightarrow a$ (resp.~$\overleftarrow a$) is both $s$-projectable and $t$-projectable. It $s$-projects onto the trivial vector field (resp.~onto $\rho_A (a)$) and $t$-projects onto $\rho_A (a)$ (resp.~onto the trivial vector field). Similarly, a direct check shows that $\overrightarrow U$ (resp.~$\overleftarrow U$) $s$-projects on the trivial $(1,1)$ tensor (resp.~onto $\rho_A \circ U$) while it $t$-projects onto $\rho_A \circ U$ (resp.~onto the trivial $(1,1)$ tensor). The third one of Formulas (\ref{eq:D,l,T,deltaU}) for $\dot{\delta U}$ now follows. The second one can be proved with a direct computation. For the first one there is an easy local proof. Namely, choose a local frame $(u_\alpha)$ of $\Gamma (A)$. Then locally
\[
U = U^\alpha \otimes u_\alpha
\]
for some local $1$-forms $U^\alpha$ on $M$. Hence
\begin{equation}\label{right_U_loc}
\overrightarrow U =  t^\ast (U^\alpha) \otimes \overrightarrow u_\alpha, \quad \text{and} \quad \overleftarrow U =  s^\ast (U^\alpha) \otimes \overleftarrow u_\alpha.
\end{equation}
A direct computation exploiting these formulas (and the obvious properties of right/left invariant vector fields) shows that
\[
\mathcal L_{\overrightarrow a} \delta U =  \overrightarrow{\mathcal L^A_{a} U}
\]
and the third one of (\ref{eq:D,l,T,deltaU}) follows (here $\mathcal L^A$ is the \emph{Lie algebroid Lie derivative}: $\mathcal L^A_a U (X) = [a, UX]_A - U[\rho_A(a), X]$, for all $a \in \Gamma (A)$ and $X \in \mathfrak X (M)$).

As for $\mathcal L_{d_A} U^\uparrow$, a straightforward computation shows that the triple $(\mathcal D, \ell, T^M)$ corresponding to it is given by the same formulas. We present (part of) it for completeness. Let $a \in \Gamma (A)$ and $D \in \Gamma (DA)$, then
\[
\begin{aligned}
\mathcal D (a) (\sigma_D)^\uparrow & = \mathcal D (a)^\uparrow (X_D) = (\mathcal L_{a^\uparrow} \mathcal L_{d_A} U^\uparrow) X_D  = (\mathcal L_{[a^\uparrow, d_A]} U^\uparrow )X_D = (\mathcal L_{X_{[a, -]_A}} U^\uparrow) X_D\\
& = \big[ X_{[a, -]}, U^\uparrow X_D \big] - U^\uparrow \left[X_{[a, -]_A}, X_D \right] =\big[ X_{[a, -]}, U (\sigma_D)^\uparrow \big] - U^\uparrow \left[X_{[[a, -]_A, D]]} \right]  \\
& = [a, U(\sigma_D)]_A^\uparrow - U(\sigma_{[[a, -], D]]})^\uparrow = \big([a, U(\sigma_D)]_A - U[\rho_A (a), \sigma_D] \big)^\uparrow \\
& = (\mathcal L_a U)(\sigma_D)^\uparrow
\end{aligned}
\]
where we used that there are no nontrivial $(1,1)$ tensors on $A[1]$ of internal degree $-2$. The rest is straightforward. 
\end{proof}

\begin{remark}
As an application of Lemma \ref{lem:delta_U} we discuss a \emph{canonical cohomology class} in the cochain complex $C^\bullet_{\mathrm{def}}(G, T^{1,0})$. Namely, let $G \rightrightarrows M$ be a Lie groupoid with Lie algebroid $A$. The identity endomorphism $\mathbb I_G : TG \to TG$ is clearly a multiplicative $(1,1)$ tensor, hence a distinguished $0$-cocycle in $C^\bullet_{\mathrm{def}}(G, T^{1,0})$. Its cohomology class $[\mathbb I_G]$ is therefore a canonical cohomology class attached to the Lie groupoid $G$. It is easy to see that $[\mathbb I_G] = 0$ if and only if $G$ integrates the tangent bundle $TM \Rightarrow M$. Indeed, let $\mathbb I = \delta U$ for some vector bundle map $U : TM \to A$. The triple $(\mathcal D, \ell, T^M)$ corresponding to $\dot{\mathbb I}$ is $(0, \mathbb I_A, \mathbb I_M)$, where $\mathbb I_A : A \to A$ and $\mathbb I_M : TM \to TM$ are the identity endomorphisms. It immediately follows from Lemma (\ref{lem:delta_U}) that $U$ is an isomorphism and $\rho_A$ is its inverse. Hence $\rho_A : A \to TM$ is a Lie algebroid isomorphism. Conversely, let $A = TM$ (with the commutator as Lie bracket, and the identity as anchor). If $G = M \times M$ is the pair groupoid, a direct computation shows that $\mathbb I_G = \delta \mathbb I_M$. In all other cases the anchor map $(s, t) : G \to M \times M$ is a locally invertible groupoid map (covering the identity map). In particular, the pull-back along $(s,t)$ is a cochain map $(s, t)^\ast : C^\bullet_{\mathrm{def}}(M \times M, T^{1,0}) \to C^\bullet_{\mathrm{def}}(G, T^{1,0})$. We conclude that
\[
\mathbb I_G = (s,t)^\ast \mathbb I_{M \times M} = (s,t)^\ast \delta \mathbb I_M = \delta \big((s,t)^\ast \mathbb I_M \big)= \delta \mathbb I_M.
\]
\end{remark}

\section{Nijenhuis Structures, Lie Algebroids, and Graded Manifolds}\label{Sec:2}

In this section we describe Nijenhuis operators in terms of DG manifolds.  It is often the case that an additional structure on a Lie algebroid $A \Rightarrow M$ is encoded by an appropriate structure on the associated $\mathbb N Q$-manifold $(A[1], d_A)$. For instance, when $A = (T^\ast M)_\pi$ is the cotangent algebroid of a Poisson manifold $(M, \pi)$ then $(A[1], d_A)$ is additionally equipped with a symplectic structure $\omega$ of internal degree $1$ such that $\mathcal L_{d_A} \omega = 0$ and every degree $1$ symplectic $\mathbb N Q$-manifold arises in this way up to isomorphisms, see \cite{R2002} (similar results hold for the Lie algebroid of a Jacobi structure \cite{G2013,M2013} and, more generally, for Lie algebroids equipped with IM vector valued forms \cite{V2016,PSV2021}). We want to prove an analogous result for Lie algebroids defined by Nijenhuis operators. We begin discussing vector bundles $A \to M$ equipped with a vector bundle isomorphism $A \cong TM$. 

Recall that an \emph{almost tangent structure} on a manifold $P$ is a $(1,1)$-tensor $V \in \Omega^1 (P, TP)$ such that $\ker P = \operatorname{im} P$ (in particular, $\dim P = \mathrm{even}$). An almost tangent structure $V$ is \emph{integrable} if $[V, V]^{\mathrm{fn}} = 0$. In other words an almost tangent structure is integrable if it is additionally a Nijenhuis operator. Let $M$ be a manifold. The vertical endomorphism $V :TTM \to TTM$ on the tangent bundle is an integrable almost tangent structure, and every integrable almost tangent structure is locally of this form. We will show in a moment that degree $-1$ almost tangent structures on $\mathbb N$-manifolds of degree $1$ are all \emph{globally} of this form in an appropriate sense. To see this, first notice that the vertical endomorphism on $TM$ is a core $(1,1)$ tensor (corresponding to the identity map $\mathbb I_M : TM \to TM$), hence it also corresponds to a $(1,1)$ tensor of degree $-1$ on $T[1]M$. Denote the latter by $V^{[1]}$, and call it the \emph{vertical endomorphism again}. If $x^i$ are local coordinates on $M$ and $\dot x{}^i$ are the associated degree $1$ fiber coordinates on $T[1] M$, then locally
\[
V = dx^i \otimes \frac{\partial}{\partial \dot x{}^i},
\]
In particular, $V$ is an integrable almost tangent structure. In is easy to see that this almost tangent structure enjoys the following ``universal property''. A vector bundle map $U : TM \to A$ induces a smooth map $T[1]M \to A[1]$ of graded manifolds, also denoted $U$. The vertical endomorphism $V$ on $A[1]$ and the $(1,1)$ tensor $U^\uparrow$ of degree $-1$ corresponding to $U$ are then automatically $U$-related.

\begin{lemma}
Let $A \to M$ be a vector bundle. The assignment $U \mapsto U^\uparrow$ establishes a bijection between vector bundle isomorphisms $U : TM \to A$ and integrable almost tangent structures of internal degree $-1$ on $A[1]$.
\end{lemma}

\begin{proof}
Begin with a vector bundle isomorphism $U : TM \to A$. According to the remark preceding the statement, the $(1,1)$ tensor $U^\uparrow \in \Omega^1 (A[1], TA[1])$ corresponding to $U$ is also the push-forward of the vertical endomorphism $V \in \Omega^1 (T[1]M, TT[1]M)$ along the diffeomorphism $U : T[1]M \to A[1]$, hence it is an integrable almost tangent structure (of internal degree $-1$) itself. Conversely, take an integrable almost tangent structure of internal degree $-1$ on $A[1]$. It is of the form $U^\uparrow$ for some vector bundle map $U : TM \to A$, and $U$ is necessarily an isomorphism. Indeed, $U$ is the composition
\[
TM \longrightarrow T_M A[1] \overset{U^\uparrow}{\longrightarrow} T_M A[1] \longrightarrow A[1] \longrightarrow A,
\]
where $T_M A[1]$ is the restriction of $TA[1]$ to the zero section of $A[1] \to M$, the map $T_M A[1] \longrightarrow A[1] $ is the natural projection and the last arrow is the shift. But, for degree reasons, $A[1] \hookrightarrow T_M A[1]$ must be in the kernel of $U^\uparrow$, hence in its image and in the image of $U$ as well. We conclude that $U$ is surjective. The injectivity now follows by dimension reasons.
\end{proof}

In what follows, given an $\mathbb N$-manifold $\mathcal M$ of degree $1$ equipped with an almost tangent structure $V$ of degree $-1$, we will always identify $\mathcal M$ with a shifted tangent bundle $T[1] M$ and $V$ with the vertical endomorphism, unless otherwise stated. For instance, we will denote by $d_{dR}$ again the push forward of the usual De Rham differential on $T[1] M$ along the diffeomorphism $T[1] M \to \mathcal M$ induced by $V$. We are now ready to state the main result of this section.

\begin{theorem}\label{theor:L_dV}
A Lie algebroid $A \Rightarrow M$ is isomorphic to the Lie algebroid $(TM)_{\mathcal N}$ induced by a Nijenhuis operator $\mathcal N \in \Omega^1 (M, TM)$ on $M$ if and only if there exists a, necessarily integrable, almost tangent structure $V$ of internal degree $-1$ on $A[1]$ such that
\begin{equation}\label{eq:[V,[V,Q]]}
[[d_A, V]^{\mathrm{fn}}, V]^{\mathrm{fn}} = 0.
\end{equation}
In this case
\begin{equation}\label{eq:L_d_A V}
[d_A, V]^{\mathrm{fn}} = \mathcal N^{\mathrm{tan}}_{[1]}
\end{equation}
where $\mathcal N \in \Omega^1 (M, TM)$ is exactly the composition $\rho_A \circ U : TM \to TM$, $U : TM \to A$ is the vector bundle isomorphism such that $U^\uparrow = V$, and we are using the latter to identify $A[1]$ with $T[1]M$ (in order to take the tangent lift).
\end{theorem}

\begin{proof}
We stress that, as the total degree of $V\in \Omega^1 (A[1], TA[1])$ is $0$, the identity (\ref{eq:[V,[V,Q]]}) is not trivial. Now, let $\mathcal N$ be a Nijenhuis operator on $M$, and let $(TM)_{\mathcal N}$ be the associated Lie algebroid structure on $TM$. We denote by $d_{\mathcal N}$ the corresponding homological vector field on $T[1] M$. We first prove (\ref{eq:L_d_A V}). We have to prove that
\begin{equation}\label{eq:L_dV}
\mathcal L_{d_{\mathcal N}} V = \mathcal N^{\mathrm{tan}}_{[1]}.
\end{equation}
As both sides are $(1,1)$ tensors of internal degree $0$ on $T[1]M$, it is enough to show that they correspond to the same triple $(\mathcal D, \ell, T^M)$. From Example \ref{ex:N^tan} we have to show that the triple $(\mathcal D, \ell, T^M)$ corresponding to the left hand side is $([-,\mathcal N]^{\mathrm{fn}}, \mathcal N, \mathcal N)$. This easily follows from Lemma \ref{lem:delta_U} by using that $V = \mathbb I_M^\uparrow$. However, beware that, although, in this case, $a$ in (\ref{eq:D,l,T,deltaU}) is an ordinary vector field, the Lie derivative $\mathcal L^A$ appearing therein is not the ordinary Lie derivative (but the Lie algebroid one). In this particular case they agree, indeed for all $Y, X \in \mathfrak X (M)$
\begin{equation}\label{eq:comput}
\begin{aligned}
\mathcal D (Y) (X) & = \big(\mathcal L_Y^{(TM)_{\mathcal N}} \mathbb I_M\big) X =  [Y, \mathbb I_M X]_{\mathcal N} - \mathbb I_M [\mathcal NY, X] \\ & = [\mathcal NY, X] + [Y, \mathcal N X] - \mathcal N [ Y, X] - [\mathcal NY, X] \\
&  = (\mathcal L_Y \mathcal N)(X)
 = [Y, \mathcal N]^{\mathrm{fn}} (X).
\end{aligned}
\end{equation}
We still need to show that $[[d_{\mathcal N}, V]^{\mathrm{fn}}, V]^{\mathrm{fn}} = 0$. This follows from (\ref{eq:L_dV}) and the fact that the Fr\"olicher-Nijenhuis bracket of the vertical endomorphism $V$ and the tangent lift of a $(1,1)$ tensor always vanishes identically (a long but straightforward computation, e.g., in local coordinates). This concludes the ``only if part'' of the proof.

Conversely, let $A \Rightarrow M$ be a Lie
 algebroid and let $V$ be an almost tangent 
 structure of internal degree $-1$ on $A[1]$. 
 We use it to identify $A$ with the tangent 
 bundle $TM$ and $V$ with the vertical 
 endomorphism. In this way, the anchor $\rho_A$ identifies with a $(1,1)$ tensor $\mathcal N$. Now, first check that (\ref{eq:[V,[V,Q]]}) implies (\ref{eq:L_dV}) (use again local coordinates). Finally, reverse the computations (\ref{eq:comput}), and use Lemma \ref{lem:delta_U}, to see that (\ref{eq:L_dV}) actually implies $[-,-]_A = [-,-]_{\mathcal N}$. This concludes the proof.
\end{proof}

There is an elegant alternative way to formulate Theorem \ref{theor:L_dV}. Namely, condition (\ref{eq:L_dV}) is actually equivalent to the simpler condition
\[
[d_{dR}, d_A] = 0.
\]
To see this, remember that vector fields on $T[1]M$ are derivations of the graded algebra $\Omega^\bullet (M)$ of differential forms on $M$. Derivations commuting with the De Rham differential are exactly those of the form $\mathcal L_K$ for some vector valued form $K \in \Omega^\bullet (M, TM)$. We conclude that a degree $1$ derivation $Q$ of $\Omega^\bullet (M)$ satisfies $[d_{dR}, Q] = 0$ if and only if $Q = \mathcal L_{\mathcal N}$ for some $(1,1)$ tensor $\mathcal N \in \Omega^1 (M, TM)$. In this case, from the properties of the Fr\"olicher-Nijenhuis bracket, $Q$ is a homological derivation if and only if $\mathcal N$ is a Nijenhuis operator, in which case $(T[1]M, Q)$ is also the $\mathcal N Q$-manifold of degree $1$ corresponding to the Lie algebroid $(TM)_{\mathcal N}$. This proves the following

\begin{theorem}\label{theor:L_dV}
A Lie algebroid $A \Rightarrow M$ is isomorphic to the Lie algebroid $(TM)_{\mathcal N}$ induced by a Nijenhuis operator $\mathcal N \in \Omega^1 (M, TM)$ on $M$ if and only if there exists a, necessarily integrable, almost tangent structure $V$ of internal degree $-1$ on $A[1]$ such that
\[
[d_{dR}, d_A] = 0.
\]
\end{theorem}

We conclude this section noticing that there is a characterization of the De Rham differential on an $\mathbb N$-manifold of degree $1$ equipped with an almost tangent structure of internal degree $-1$. Namely, we have the following proposition that might be of independent interest.

\begin{proposition}
The De Rham differential $d_{dR}$ is the only homological vector field on $T[1] M$ such that 
\begin{equation}\label{eq:i_dRV}
\iota_{d_{dR}} V = \mathcal E
\end{equation}
where $\mathcal E \in \mathfrak X (T[1] M)$ is the Euler vector field.
\end{proposition}

\begin{proof}
Formula (\ref{eq:i_dRV}) can be proved easily, e.g.~in local coordinates. Now, suppose that there exists another homological vector field $Q$ on $T[1]M$ such that 
\begin{equation}\label{eq:i_QV}
\iota_Q V = \mathcal E.
\end{equation}
Then $Q$ gives to $TM$ the structure of a Lie algebroid whose anchor $\rho : TM \to TM$ is the identity $\mathbb I_M$. Indeed, locally
\[
Q = \rho^{i}_j \dot x{}^j \frac{\partial}{\partial x^i} - \frac{1}{2}c_{ij}^k \dot x^i \dot x^j \frac{\partial}{\partial \dot x{}^k}
\]
and the condition (\ref{eq:i_QV}) reads $
\rho^{i}_j = \delta^i_j$, i.e.~$\rho = \mathbb I_M$. As the anchor is a Lie algebroid map with values in the standard tangent bundle Lie algebroid, $Q$ must be the De Rham differential.
\end{proof}

%\section{Nijenhuis Structures and IM Tensors}

\section{Nijenhuis Structures and Lie Groupoids}\label{Sec:3}

Let $G \rightrightarrows M$ be a Lie groupoid and let $A \Rightarrow M$ be its Lie algebroid. We say that $G$ \emph{integrates} a Nijenhuis operator $\mathcal N \in \Omega^1 (M, TM)$ on $M$ if there exists a Lie algebroid isomorphism $A \cong (TM)_{\mathcal N}$. In the next theorem we characterize Lie groupoids integrating a Nijenhuis operator. 

\begin{theorem}\label{theor:delta_U}
A Lie groupoid $G \rightrightarrows M$ with Lie algebroid $A \Rightarrow M$ integrates a Nijenhuis operator $\mathcal N \in \Omega^1 (M, TM)$ on $M$ if and only if there exists a vector bundle map $U : TM \to A$ such that 
\begin{enumerate}
\item $\ker \overrightarrow U = T^t G$ and $\operatorname{im} \overrightarrow U = T^s G$,
\item $[\overrightarrow U, \overrightarrow U]^{\mathrm{fn}} = 0$, i.e. $\overrightarrow U$ is a Nijenhuis operator on $G$.
\end{enumerate}
In this case, put $\mathcal N = s_\ast \delta U = t_\ast \delta U$. Then $\mathcal N$ is a Nijenhuis operator on $M$, and $U : A \to (TM)_{\mathcal N}$ is a Lie algebroid isomorphism. Finally, if we use $U$ to identify $A$ and $(TM)_{\mathcal N}$, we also have
\begin{equation}\label{eq:delta_U}
\dot{\delta U} = \mathcal N^{\mathrm{tan}}.
\end{equation}
\end{theorem}

\begin{proof}
Let's first assume that $G$ integrates a Nijenhuis operator $\mathcal N \in \Omega^1 (M, TM)$, so that we can identify $A$ and $(TM)_{\mathcal N}$. Let $U = \mathbb I_M : TM \to TM$ be the identity. Then, locally
\[
\overrightarrow U = t^\ast (dx^i) \otimes \overrightarrow{\partial_i},
\]
for some local coordinates $(x^i)$ on $M$, where we put $\partial_i := \partial / \partial x^i$. As $\dim G = 2 \dim M$, point \emph{(1)} immediately follows. Next, compute
\[
\begin{aligned}
\big[\overrightarrow U, \overrightarrow U\big]^{\mathrm{fn}} & = \big[t^\ast (dx^i) \otimes \overrightarrow{\partial_i}, t^\ast (dx^j) \otimes \overrightarrow{\partial_j} \big]^{\mathrm{fn}}\\
& = t^\ast (dx^i \wedge dx^j) \otimes  \big[\overrightarrow{\partial_i},\overrightarrow{\partial_j} \big]^{\mathrm{fn}} + 2 t^\ast (dx^i) \wedge \mathcal L_{\overrightarrow{\partial_i}} t^\ast (dx^k) \otimes \overrightarrow{\partial_k},
\end{aligned}
\]
where we used standard properties of the Fr\"olicher Nijenhuis bracket (together with the fact that the coordinate $1$-forms are closed). Now we have
\[
\mathcal N \partial_i = \mathcal N^k_i \partial_k \quad \text{and} \quad \left[\partial_i, \partial_j\right]_{\mathcal N} = c_{ij}^k \partial_k = \left(\partial_i \mathcal N^k_j - \partial_j \mathcal N^k_i \right) \partial_k,
\]
for some local functions $\mathcal N^k_i, c_{ij}^k$. As 
\[
t_\ast \overrightarrow{\partial_i} = \mathcal N \partial_i =  \mathcal N^k_i \partial_k, 
\]
it follows that
\[
\big[\overrightarrow U, \overrightarrow U\big]^{\mathrm{fn}} = t^\ast \left(c_{ij}^k dx^i \wedge dx^j + 2 dx^i \wedge d\mathcal N_i^k \right) \otimes \overrightarrow{\partial_k} = \left( c_{ij}^k - \partial_i \mathcal N^k_j + \partial_j \mathcal N^k_i \right)  \otimes \overrightarrow{\partial_k} = 0,
\]
as claimed. A very similar computation shows that $[\overleftarrow U, \overleftarrow U]^{\mathrm{fn}} = 0$ as well. It is even easier to check that $[\overrightarrow U, \overleftarrow U]^{\mathrm{fn}} = 0$. It follows that 
\[
\big[\delta U, \delta U\big]^{\mathrm{fn}} = \big[\overrightarrow U + \overleftarrow U, \overrightarrow U + \overleftarrow U\big]^{\mathrm{fn}} =  \big[\overrightarrow U, \overrightarrow U\big]^{\mathrm{fn}} + 2\big[\overrightarrow U, \overleftarrow U\big]^{\mathrm{fn}} + \big[\overleftarrow U, \overleftarrow U\big]^{\mathrm{fn}} = 0.
\]
Formula (\ref{eq:delta_U}) now follows from Lemma \ref{lem:delta_U} (and Example \ref{ex:N^tan}).

Conversely, let $U : TM \to A$ be a vector bundle map satisfying \emph{(1)} and \emph{(2)}. It follows from $\ker \overrightarrow U = \ker T^t G$, resp.~$\operatorname{im} \overrightarrow U = \ker T^s G$, by restriction to the units, that $U$ is injective, resp.~surjective. Hence $U$ is a vector bundle isomorphism that we can use to identify $A$ with $TM$, as vector bundles. Similarly, we identify the anchor map $\rho_A : A \to TM$ with a $(1,1)$ tensor $\mathcal N \in \Omega^1 (M, TM)$. From \emph{(2)}, reversing the above coordinate computation, one now sees that $A = (TM)_{\mathcal N}$.
\end{proof}

\begin{remark}\label{rem:A-torsion}
Let $G \rightrightarrows M$ be a Lie groupoid, let $A \Rightarrow M$ be its Lie algebroid, and let $U : TM \to A$ be a vector bundle map. Using, e.g.~(\ref{right_U_loc}), we immediately see that condition \emph{(1)} in Theorem \ref{theor:delta_U} is actually equivalent to $U$ being a vector bundle isomorphism. Condition \emph{(2)} can also be expressed purely in terms of $U$ (and independently of condition \emph{(1)}). To do this, we define the \emph{$A$-torsion} of $U$ to be $A$-valued $2$-form $T^A_U \in \Omega^2 (M, A)$ given by
\[
T^A_U (X, Y) := [UX, UY]_A + U\rho_AU[X,Y] - U[\rho_AUX, Y] -U [X, \rho_AUY].
\]
Now, given an $A$-valued $2$-form $T \in \Omega^2 (M, A)$, the obvious adaptation of Formulas (\ref{eq:arrow_U}) defines vector valued $2$-forms $\overrightarrow T, \overleftarrow T \in \Omega^2 (G, TG)$ on $G$. We also have that $\overrightarrow T|_{TM} = T$ and $\overleftarrow T|_{TM} = di \circ T$. Finally, a long but easy computation exploiting, e.g., (\ref{right_U_loc}), shows that
\[
T_{\overrightarrow U} = \overrightarrow{T^A_U}, \quad \text{and} \quad T_{\overleftarrow U} = \overleftarrow{T^A_U}
\]
We leave the details to the reader. We conclude that condition \emph{(2)} of Theorem \ref{theor:delta_U} is equivalent to
\[
T^A_U = 0.
\]
\end{remark}

We conclude this section discussing a few elementary examples.

\begin{example}[Trivial Nijenhuis operators]
Let $G$ integrate a Njenhuis operator $\mathcal N \in \Omega^1 (M, TM)$, and let $U$ be as in the statement of Theorem \ref{theor:delta_U}. Clearly, $\overrightarrow U \in \Omega^1 (G, TG)$ is not an almost tangent structure, unless $\ker T^s G = \ker T^t G$ which happens exactly when $s = t$, i.e.~$G$ is a bundle of Lie groups. But in this case, $\mathcal N = 0$ necessarily, so $G$ is a bundle of Lie groups integrating the trivial Lie algebroid structure $(TM)_0$ on $TM$. The source simply connected such $G$ is $TM \rightrightarrows M$ with both source and target being the canonical projection $TM \to M$, and the multiplication being the fiber-wise addition. We denote this bundle of abelian Lie groups $(TM)_+$. If $U : TM \to TM$ is the identity map then $\overrightarrow U = V \in \Omega^1 (TM, TTM)$ is the vertical endomorphism (use, e.g., local coordinates) and $\overleftarrow U = - V$, so that $\delta U = \overrightarrow U + \overleftarrow U = 0$ (also in agreement with (\ref{eq:delta_U})).

On the other hand let $G$ be a source connected \emph{proper} Lie groupoid integrating $(TM)_0$. Actually $G$ is a torus bundle and there must be a Lie groupoid map $p : (TM)_+ \to G$. But 1) $p$ is a local diffeomorphism, 2) the vertical endomorphism $V$ on $(TM)_+$ descends to $G$ and 3) $\Lambda := \ker p$ is a lattice in $TM$. So translations along sections of $\Lambda$ preserve $V$. But, actually, $V$ is preserved by a translation along any section of $TM \to M$. It follows that every lattice in $TM$ arises in this way. We conclude that the present situation is very different from that of the integration of the trivial Poisson structure by a proper symplectic groupoid which gives very special lattices in $T^\ast M$ (hence in $TM$), namely those corresponding to integral affine structures on $M$ (see \cite{PMCTII} for details).
\end{example}

\begin{example}[Invertible Nijenhuis operators]
Let $M$ be a manifold and let $\mathcal N$ be an invertible Nijenhuis operator on $M$ (e.g.~a complex structure). Then $\mathcal N$ itself is an isomorphism identifying $(TM)_{\mathcal N}$ and  $(TM)_{\mathbb I}$ where $\mathbb I : TM \to TM$ is the identical $(1,1)$ tensor. In the following we will understand this identification. The Lie algebroid $(TM)_{\mathbb I}$ is the usual tangent Lie algebroid which is integrated (among others) by the pair groupoid $M \times M \rightrightarrows M$. The identity $\mathbb I : TM \to TM$ can also be seen as a $(TM)_{\mathbb I}$-valued $1$-form on $M$. If we do so, then   
\[
\overrightarrow{\mathbb I}, \overleftarrow{\mathbb I} : T(M \times M) = TM \times TM \to T(M \times M) = TM \times TM
\]
are the projections onto the first and the second factor respectively, so that $\delta \mathbb I = \overrightarrow{\mathbb I} + \overleftarrow{\mathbb I}$ is the identity of $T(M \times M)$ (in agreement with (\ref{eq:delta_U})).
\end{example}

\begin{example}[Rank $1$ Lie algebroids on a $1$ dimensional manifold]
Let $M$ be either the line $\mathbb R$ or the circle $S^1$ and let $\theta$ be the canonical coordinate on $M$. Any Lie algebroid structure $A$ on the trivial line bundle $\mathbb R_M$ is of the following type:
\[
[f, g]_A = f X(g) - X(g) f, \quad \rho_A (f) = fX,\quad f,g \in \Gamma (A) = C^\infty (M),
\]
for some vector field $X \in \mathfrak X (M)$. Denote $F := X (\theta)$, so that $X = F \frac{\partial}{\partial \theta}$. The vector bundle isomorphism
\begin{equation}\label{eq:iso_1}
\mathbb R_M \to TM, \quad f \mapsto f \frac{\partial}{\partial \theta}
\end{equation}
identifies $A$ with the Lie algebroid $(TM)_{\mathcal N}$, where $\mathcal N$ is the Nijenhuis operator given by
\[
\mathcal N = d\theta \otimes X = F \mathbb I.
\]
Now, in order to illustrate Theorem \ref{theor:delta_U}, we take the long route to prove this latter thing. Recall that $A$ is integrated by the Lie groupoid $D^X \rightrightarrows M$, where $D^X \subseteq \mathbb R \times M$ is the domain of the flow of $X$:
\[
\phi^X : D^X \to M, \quad (\varepsilon, \theta) \mapsto \phi^X_\varepsilon (\theta).
\]
The source $s : D^X \to M$ is the projection onto the second factor, while the target $t$ is $\phi^X$. Two arrows $(\bar \varepsilon, \bar \theta), (\varepsilon, \theta) \in D^X$ are composable when $\bar \theta = \phi^X_{\varepsilon} (\theta)$ and, in this case, their product is
\[
(\bar \varepsilon, \bar \theta) \cdot (\varepsilon, \theta) = (\bar \varepsilon + \varepsilon, \theta).
\]
The inversion $i : D^X \to D^X$ maps $(\varepsilon, \theta)$ to $(-\varepsilon, \phi_\varepsilon^X (\theta))$.

The inverse of the isomorphism (\ref{eq:iso_1}) is
\[
U = d\theta \otimes u : TM \to \mathbb R_M
\]
where we denoted by $u$ the constant function $1$. A straightforward computation shows that
\[
\overrightarrow U = t^\ast (d\theta) \otimes \overrightarrow u = d\phi^X \otimes \frac{\partial}{\partial \varepsilon} \quad \text{and} \quad \overleftarrow U = s^\ast (d\theta) \otimes \overleftarrow u = d\theta \otimes \left(i^\ast \left( \frac{\partial \phi^X}{\partial \varepsilon}\right) \frac{\partial}{\partial \theta} - \frac{\partial}{\partial \varepsilon} \right).
\]
Finally, using that 
\[
\frac{\partial \phi^X}{\partial \varepsilon} = F \circ \phi^X,
\]
we find
\[
\delta U = \overrightarrow U + \overleftarrow U = (F \circ \phi^X) \,d\varepsilon \otimes \frac{\partial}{\partial \varepsilon} + \left(\frac{\partial \phi^X}{\partial \theta} - 1 \right) dx \otimes \frac{\partial}{\partial \theta}+ F d\theta \otimes \frac{\partial}{\partial \theta}
\]
 which is readily seen to be a Nijenhuis operator projecting to 
 \[
 \mathcal N = F d\theta \otimes \frac{\partial}{\partial \theta} = F \mathbb I
 \]
 under both $s$ and $t$. It follows that $U : (TM)_{\mathcal N} \to A$ is a Lie agebroid isomorphism as already noticed.
 \end{example}
 
 \section{More Examples}\label{Sec:4}
 
 In this section we discuss some slightly less trivial examples of Lie groupoids integrating a Nijenhuis operator, including their multiplicative Nijenhuis structure.
 
 \subsection{The vertical endomorphism of the tangent bundle}\label{Subsec:5}
 
 Let $M$ be a manifold and let $V \in \Omega^1 (M, TM)$ be an integrable almost tangent structure on $M$. In particular, $V$ is a Nijenhuis operator, and we have a Lie algebroid $(TM)_V \Rightarrow M$. The local model for an integrable almost tangent structure is the vertical endomorphism of the tangent bundle. For simplicity, we assume that $M = TB$ globally for some manifold $B$, and that $V$ is exactly the vertical endomorphism. In this case, the Lie algebroid $(TM)_V \Rightarrow M$ is integrated by a Lie groupoid $G \rightrightarrows M$ (depending on $B$ only) that we now describe. As a manifold $G = TM = TTB$ (the double tangent bundle of $B$). To the best of our knowledge the following groupoid structure on $TTB$ appears here for the first time. In order to describe it, we recall a few properties of the double tangent bundle. First of all, it is a double vector bundle.
 \[
 \begin{array}{c}
 \xymatrix{
 TTB \ar[r]^{\tau'} \ar[d]_-{\tau}& TB \ar[d]\\
   TB \ar[r] & B
 }
 \end{array}.
 \] 
 The vertical projection $\tau : TTB \to TB$ is the usual tangent bundle projection mapping a tangent vector to its base point. The horizontal projection $\tau' : TTB \to TB$ is the tangent to the projection $TB \to B$. We denote by $(+, \cdot)$ the fiber-wise operations (addition and scalar multiplication) of the vector bundle with projection $\tau$, and by $(+', \cdot')$ the fiber-wise operations of the vector bundle with projection $\tau'$. The latter are the tangent to the fiber-wise operations of the vector bundle $TB \to B$. Given local coordinates $z = (z^i)$ on $B$, we denote by $(z, \dot z)$ the associated tangent coordinates on $TB$, and by $(z, \dot z, z', \dot z{}')$ the tangent coordinates on $TTB$ associated to the coordinates $(z, \dot z)$. In these coordinates we have
 \[
 \tau (z, \dot z, z', \dot z{}') = (z, \dot z), \quad \tau' (z, \dot z, z', \dot z{}') = (z, z').
 \]
 Additionally
 \[
 (x, \dot x, z', \dot z{}') + (x, \dot x, w', \dot w{}') = (x, \dot x, z' + w', \dot z{}' + \dot w{}'), \quad a \cdot  (x, \dot x, z', \dot z{}') =  (x, \dot x, az', a\dot z{}'), 
 \]
 and
 \[
  (x, \dot z, x', \dot z{}') +' (x, \dot w, x', \dot w{}') = (x, \dot z + \dot w, x', \dot z{}' + \dot w{}'), \quad b \cdot'  (x, \dot z, x', \dot z{}') =  (x, b\dot z, x', b\dot z{}'), 
 \]
 for all $a,b \in \mathbb R$. Finally there is a canonical involution $\varkappa : TTB \to TTB$ swapping the two vector bundle structures. In coordinates
 \[
 \varkappa (z, \dot z, z', \dot z{}') = (z, z', \dot z, \dot z{}').
 \]
 
 We are now ready to describe the Lie groupoid structure on $TTB$ integrating the vertical endomorphism $V$ on $TB$. Source and target $s, t : TTB \to TB$ are given by
 \[
 s(\xi) = \tau(\xi) - \tau' (\xi), \quad t (\xi) = \tau (\xi) + \tau' (\xi).
 \]
 As $\tau (\xi)$ and $\tau' (\xi)$ have the same base point for all $\xi \in TTB$, both $s$ and $t$ are well-defined. In order to define the multiplication
 \[
 m : TTB \mathbin{{}_s\times_t} TTB \to TTB,
 \]
 take $\xi, \zeta \in TTB$ such that $s(\xi ) = t (\zeta)$, and let $\eta \in TTB$ be any vector such that 
 \[
 \tau (\eta) = \tau' (\xi) \quad \text{and} \quad \tau' (\eta) = \tau' (\zeta).
 \]
 We put
 \[
m (\xi , \zeta) = \Big( \xi -' \varkappa (\eta) \Big) + \Big( \zeta +' \eta\Big).
 \]
 In coordinates
 \[
 m\Big((x, \dot z, z', \dot z{}'), (x, \dot w, w', \dot w{}')\Big) = (x, \dot z - w', z' + w', \dot z{}' + \dot w{}')
 \]
 where $\dot z - z' = \dot w + w'$. This shows that $m (\xi, \zeta)$ is independent of the choice of $\eta$. The unit $u : TB \to TTB$ is the zero section of the vertical vector bundle $\tau : TTB \to TB$. Finally, the inversion $i : TTB \to TTB$ is fiber-wise multiplication by $-1$ wrt the vertical vector bundle structure. A direct computation, e.g.~in coordinates, shows that, with these structure maps, $TTB$ is indeed a Lie groupoid over $TB$. Denote by $G \rightrightarrows TB$ this Lie groupoid. We want to show that $G$ integrates the vertical endomorphism on $TB$. To do this we need to describe the Lie algebroid $A \Rightarrow TB$ of $G$. As $\tau : G \to TB$ is a vector bundle projection and $u : TB \to G$ is the zero section of this vector bundle, we have that $A = u^\ast T^s G \cong TTB \oplus_{TB} \dot TTB$. Here $\dot TTB$ denotes the copy of $TTB$ corresponding to the tangent spaces to the $\tau$-fibers at zeros. Clearly, the map
 \[
U := \tfrac{1}{2} \big(V \oplus \mathbb I \big) : TTB \to A
 \]
 is a vector bundle isomorphism. In coordinates
 \[
 U \frac{\partial}{\partial z}|_{(z, \dot z)} = \frac{1}{2} \left( \frac{\partial}{\partial \dot z} 
 + \frac{\partial}{\partial z'} \right)\!|_{(z, \dot z, 0, 0)}, \quad \text{and} \quad U \frac{\partial}{\partial \dot z}|_{(z, \dot z)} = \frac{1}{2} \frac{\partial}{\partial \dot z{}'}|_{(z, \dot z, 0, 0)},
 \]
 i.e.
  \[
 U  = \frac{1}{2} \left(dz \otimes \left( \frac{\partial}{\partial \dot z} 
 + \frac{\partial}{\partial z'} \right)\!|_{M} +  d \dot z \otimes \frac{\partial}{\partial \dot z{}'}|_{M}\right).
 \]
 It remains to show that $U$ identifies the Lie algebroid structure $(TTB)_V$ with that of $A$. Instead of doing this directly, we apply Theorem \ref{theor:delta_U}. First of all, we compute $\overrightarrow U$. It is easy to see that 
 \begin{equation}\label{eq:134}
\overrightarrow{ \left(\frac{\partial}{\partial \dot z} 
 + \frac{\partial}{\partial z'}\right)\!|_M} = \frac{\partial}{\partial \dot z} 
 + \frac{\partial}{\partial z'} \quad \text{and} \quad \overrightarrow{\frac{\partial}{\partial \dot z{}'}|_M} = \frac{\partial}{\partial \dot z{}'}.
 \end{equation}
Now, denote by $V_{TTB} \in \Omega^1 (TTB, TTTB)$ the vertical endomorphism on $TTB$. It follows from (\ref{eq:134}) and the first one of (\ref{right_U_loc}) that
 \[
 \begin{aligned}
 \overrightarrow U & = \frac{1}{2} \left(d(t^\ast z) \otimes \left( \frac{\partial}{\partial \dot z} 
 + \frac{\partial}{\partial z'} \right) +  d (t^\ast \dot z) \otimes \frac{\partial}{\partial \dot z{}'}\right) \\
 & = \frac{1}{2} \left(d z \otimes \left( \frac{\partial}{\partial \dot z} 
 + \frac{\partial}{\partial z'} \right) +  d \big(\dot z + z'\big) \otimes \frac{\partial}{\partial \dot z{}'}\right) \\
 & = \frac{1}{2} \Big( \varkappa_\ast (V_{TTB}) + V_{TTB}\Big).
 \end{aligned}
 \]
 which clearly fulfills both conditions \emph{(1)} and \emph{(2)} in Theorem \ref{theor:delta_U}. Similarly
 \[
 \overleftarrow U = \frac{1}{2}\Big( \varkappa_\ast (V_{TTB}) - V_{TTB}\Big).
 \]
 Hence
 \[
 \delta U = \overrightarrow U + \overleftarrow U = \varkappa_\ast \big(V_{TTB}\big)
 \]
 which is a Nijenhuis operator projecting on $V_{TB}$ along both $s, t : TTB \to TB$. Using Theorem \ref{theor:delta_U} we conclude that $U$ identifies the Lie algebroid structure $(TTB)_V$ with that of $A$ as claimed.
 
 \subsection{Integrable projections}
 
 A \emph{projection} on a manifold $M$ is a $(1,1)$ tensor $P$ such that $P^2 = P$. It follows that $V := \operatorname{im} P$ and $H := \ker P$ are regular distributions (the \emph{vertical} and \emph{horizontal} distributions) such that $TM = V \oplus H$. A projection $P$ is \emph{integrable} if it is additionally a Nijenhuis operator. In this case both $V$ and $H$ are involutive distributions. Let $P$ be an integrable projection on $M$. For simplicity, we assume that the foliation integrating the vertical distribution $V$ is simple, i.e.~the leaf space $B$ is a manifold, and the natural projection $\pi: M \to B$ is a surjective submersion. In this case $V = T^\pi M$ the $\pi$-vertical tangent bundle, and $H$ is a flat Ehresmann connection on the fibration $\pi : M \to B$. Locally, we can choose fibered coordinates $(x^i, u^\alpha)$ on $M$ such that 
 \[
 P = du^\alpha \otimes \frac{\partial}{\partial u^\alpha}.
 \]
 In the following, we will always use $\pi$ to identify $H$ with the pull-back vector bundle $\pi^\ast TB$ in the obvious way. This vector bundle carries a natural representation of the Lie algebroid $T^\pi M \Rightarrow M$ (where the anchor is the inclusion $T^\pi M \hookrightarrow TM$ and the bracket is the commutator of vector fields tangent to fibers of $\pi$). This representation is given by the \emph{Bott connection}: the unique $T^\pi M$-connection in $H = \pi^\ast TB$ such that all pull-back sections are parallel. The Lie algebroid $(TM)_P \Rightarrow M$ of the integrable projection $P$ is easily seen to be isomorphic to the semi-direct product Lie algebroid $T^\pi M \ltimes H = T^\pi M \times_B TB \Rightarrow M$ under
 \begin{equation}\label{eq:U_P}
 TM \to T^\pi M \times_B TB, \quad v \mapsto (Pv, \pi_\ast v).
 \end{equation}
 In order to illustrate our main result, we now prove the Lie algebroid isomorphism $(TM)_P \cong T^\pi M \ltimes H$ using Theorem \ref{theor:delta_U}. To do this we need to fix a Lie groupoid integrating $T^\pi M \ltimes H$. First we choose an integration of the Lie algebroid $T^\pi M$. The easiest choice is the \emph{submersion groupoid} $M \times_B M$ whose structure maps are
 \[
 \begin{aligned}
 s(x,y) = x&, \quad t(x, y) = y, \quad m((x,y),(z,x)) = (z,y), \\ & u(x) = (x,x), \quad i(x,y) = (y,x).
 \end{aligned}
 \]
 The vector bundle $H = \pi^\ast TB$ carries a canonical representation of the submersion groupoid integrating the Bott connection given by
 \[
 (x, y).(x, v) = (y, v), \quad \text{for all $(x, y) \in M \times_B M$, and $v \in T_{\pi (x) = \pi(y)} B$}.
 \]
 It follows that the semidirect product Lie groupoid $(M \times_B M) \ltimes H = M \times_B M \times_B TB \rightrightarrows M$ integrates the Lie algebroid $T^\pi M \ltimes H = T^\pi M \times_B TB \Rightarrow M$. The structure maps in $M \times_B M \times_B TB$ are
 \begin{equation}\label{eq:sm_rep}
 \begin{aligned}
 s(x,y, v) = x&, \quad t(x, y, v) = y, \quad m((x,y, v),(z,x, w)) = (z,y, v + w), \\ & u(x) = (x,x, 0_x), \quad i(x,y, v) = (y,x, -v).
 \end{aligned}
 \end{equation}
 Notice that $T^\pi M \times_B TB \Rightarrow M$ identifies with the Lie algebroid 
 \[
 A \subseteq T\big( M \times_B M \times_B TB\big) = TM \times_{TB} TM \times_{TB} TTB
 \]
  of $M \times_B M \times_B TB$ under the inclusion
 \[
 T^\pi M \times_B TB \hookrightarrow TM \times_{TB} TM \times_{TB} TTB, \quad  (\xi, v) \mapsto \big(0_x, \xi, v_{0_{\pi(x)}}^\uparrow\big) 
 \]
where $x = \tau (\xi)$. Under this inclusion, the map (\ref{eq:U_P}) becomes the vector bundle isomorphism
 \[
 U : TM \to A, \quad \xi \mapsto \big(0_x, P\xi, (\pi_\ast \xi)_{0_{\pi(x)}}^\uparrow\big).
 \]
 A straightforward computation using this and the structure maps (\ref{eq:sm_rep}) shows that the $(1,1)$ tensors
 \[
 \overrightarrow U , \overleftarrow U : TM \times_{TB} TM \times_{TB} TTB \to TM \times_{TB} TM \times_{TB} TTB
 \]
 are given by
 \[
 \overrightarrow U (\xi, \eta, W) = \big(0_x, P\eta, (\pi_\ast \xi)^\uparrow_w\big), \quad \text{and} \quad  \overleftarrow U (\xi, \eta, W) = \big(P\xi, 0_y, - (\pi_\ast \xi)^\uparrow_w\big)
 \]
 where $(\xi, \eta, W)$ is a tangent vector at the point $(x, y, w) \in M \times_B M \times_B TB$. The coordinates $(x^i, u^\alpha)$ on $M$ induce coordinates $(x^i, u^\alpha_1, u^\alpha_2, \dot x{}^i)$ on $M \times_B M \times TB$, where $(u^\alpha_1)$ (resp.~$(u^\alpha_2)$) are fiber coordinates on the first (resp.~second) factor, and $(\dot x{}^i)$ are fiber tangent coordinates on the last factor. In these coordinates
 \[
 \overrightarrow U = du^\alpha_2 \otimes \frac{\partial}{\partial u^\alpha_2} + dx^i \otimes \frac{\partial}{\partial \dot x{}^i}, \quad \text{and} \quad \overleftarrow U = du^\alpha_1 \otimes \frac{\partial}{\partial u^\alpha_1} - dx^i \otimes \frac{\partial}{\partial \dot x{}^i}.
 \]
 which obviously satisfy conditions \emph{(1)} and \emph{(2)} of Theorem \ref{theor:delta_U} (alternatively one can check that the $A$-torsion of $U$ vanishes identically and then use Remark \ref{rem:A-torsion}). Finally
 \[
 \delta U = \overrightarrow U + \overleftarrow U : TM \times_{TB} TM \times_{TB} TTB \to TM \times_{TB} TM \times_{TB} TTB
 \]
 is given by
 \[
  \delta U (\xi, \eta, W) = (P\xi, P\eta, 0)
 \]
 which projects onto $P$ under both the source and the target, as claimed (equivalently $\rho_A \circ U = P)$.
 
 \subsection{Pre-Lie algebras}
 
 A \emph{pre-Lie algebra} (aka \emph{left symmetric algebra}) is a vector space $\mathfrak a$ equipped with a bilinear map
 \[
 \triangleright : \mathfrak a \times \mathfrak a \to \mathfrak a, \quad (a, b) \mapsto a \triangleright b
 \]
 such that
 \[
 (a \triangleright b) \triangleright c - a \triangleright (b \triangleright c) = (b \triangleright a) \triangleright c + b \triangleright (a \triangleright c), \quad \text{for all $a,b,c \in \mathfrak a$}.
 \]
 In other words the \emph{associator} of $\mathfrak a$ is symmetric in the first to entries. Associative algebras are instances of pre-Lie algebras. If $(\mathfrak a, \triangleright)$ is a pre-Lie algebra, then the \emph{commutator}:
 \[
 [-,-]_\triangleright : \mathfrak a \times \mathfrak a \to \mathfrak a, \quad (a, b) \mapsto [a, b]_\triangleright
 \]
 is a Lie bracket. We will denote by $\mathfrak a_{\mathrm{Lie}}$ the Lie algebra $(\mathfrak a, [-,-]_\triangleright)$. The Lie algebra $\mathfrak a_{\mathrm{Lie}}$ comes with a canonical representation $L : \mathfrak a_{\mathrm{Lie}} \to \mathfrak{gl} (\mathfrak a)$ on $\mathfrak a$ itself given by
 \[
 L(a) (x) = a \triangleright x, \quad \text{for all $a, x \in \mathfrak a$}.
 \]
 
 Let $\mathfrak a$ be a finite dimensional real pre-Lie algebra. Then $\mathfrak a$ can be seen as a manifold. For any $v \in \mathfrak a$, we denote by $a^\uparrow \in \mathfrak X (\mathfrak a)$ the constant vector field equal to $a$ and, for any endomorphism $\phi : \mathfrak a \to \mathfrak a$, we denote by $X_\phi \in \mathfrak X (\mathfrak a)$ the only vector field such that $[X_\phi, a^\uparrow] = \phi (a)^\uparrow$ for all $a \in \mathfrak g$ (this is consistent with our notation in Section \ref{Sec:2}). There is a canonical $(1,1)$ tensor $\mathcal N$ on $\mathfrak a$ given by
 \[
 \mathcal N a^\uparrow = X_{L(a)}, \quad \text{for all $a \in \mathfrak a$},
 \]
Equivalently
\[
\mathcal N a_x^\uparrow = - a \triangleright x, \quad \text{for all $a, x \in \mathfrak a$}. 
\]
The $(1,1)$ tensor $\mathcal N$ is always a Nijenhuis operator and every Nijenhuis operator on a vector space whose components are linear in linear coordinates arises in this way \cite{BKM2022, K2019}. We remark that our Nijenhuis operator $\mathcal N$ differs in sign from that in \cite{BKM2022, K2019}. Our sign conventions make it easier the description of the associated Lie algebroid.

We will describe the source-simply connected Lie groupoid integrating $\mathcal N$. We begin describing $(T\mathfrak a)_{\mathcal N}$. For all $a, b \in \mathfrak a$ we have
\[
\begin{aligned}
{}[a^\uparrow, b^\uparrow]_{\mathcal N} & = [\mathcal N a^\uparrow, b^\uparrow] + [a^\uparrow, \mathcal N b^\uparrow] = [X_{L(a)}, b^\uparrow] + [a^\uparrow, X_{L(b)}] \\ 
&= L(a)(b) - L(b)(a) = [a, b]_\triangleright.
\end{aligned}
\]
The latter equation, together with the anchor being $\mathcal N$, determines the Lie algebroid $(T\mathfrak a)_{\mathcal N}$ completely. Namely, it is immediate that $(T\mathfrak a)_{\mathcal N}$ is isomorphic to the action Lie algebroid $A := \mathfrak a_{\mathrm{Lie}} \ltimes \mathfrak a \Rightarrow \mathfrak a$ via
\[
U : T\mathfrak a \to \mathfrak a \times \mathfrak a, \quad v^\uparrow_x \to (v, x).
\]

Again, in order to illustrate the main result of the paper, we prove this straightforward fact taking a longer route. So, let $G$ be the simply connected Lie group integrating $\mathfrak a_{\mathrm{Lie}}$. The pre-Lie algebra structure on $\mathfrak a$ induces a left invariant affine structure on $G$ and every source-simply connected Lie group with a left invariant affine structure arises in this way (see, e.g.~\cite{B1996}). We will not really use this affine structure, but we will need the $G$-action $\mathcal L : G \times \mathfrak a \to \mathfrak a$ on $\mathfrak a$ integrating the Lie algebra action $L$.

\begin{lemma}\label{lem:final}
Let $(\mathfrak a, \triangleright)$ be a finite dimensional, real pre-Lie algebra and let $G$ be the source simply-connected Lie group integrating $\mathfrak a_{\mathrm{Lie}}$. Then the action $\mathcal L$ satisfies
\begin{equation}\label{eq:calL}
\mathcal L_g (x \triangleright y) = \operatorname{ad}_g x \triangleright \mathcal L_g y, \quad \text{for all $x, y \in \mathfrak a$}.
\end{equation}
\end{lemma}

\begin{proof}
By connectedness, it is enough to prove that (\ref{eq:calL}) is satisfied at the infinitesimal level. Differentiating $\mathcal L_g (x \triangleright y) - \operatorname{ad}_g x \triangleright \mathcal L_g y$ we find
\[
\dot g \triangleright (x \triangleright y) - [\dot g, x]_\triangleright \triangleright y -  x \triangleright (\dot g \triangleright y) = \dot g \triangleright (x \triangleright y) - (\dot g \triangleright x) \triangleright y + (x \triangleright \dot g)\triangleright y - x \triangleright (\dot g \triangleright y) = 0,
\]
for all $\dot g, x, y \in \mathfrak a$. This concludes the proof.
\end{proof}

We want to show that the action groupoid $G \ltimes \mathfrak a \rightrightarrows \mathfrak a$ corresponding to $\mathcal L$ integrates $\mathcal N$ via Theorem \ref{theor:delta_U}. Recall that the structure maps of $G \ltimes \mathfrak a$ are:
\begin{equation*}
 \begin{aligned}
 s(g, x) = x&, \quad t(g, x) = \mathcal L_g x, \quad m((g, \mathcal L_h x),(h, x)) = (gh, x), \\ & u(x) = (1_G, x), \quad i(g, x) = (g^{-1}, \mathcal L_g x).
 \end{aligned}
 \end{equation*}
The Lie algebroid of $G \times \mathfrak a \to \mathfrak a$ is the action Lie algebroid $A$ that we already considered. We want to show that $U$ satisfies both conditions \emph{(1)} and \emph{(2)} of Theorem \ref{theor:delta_U} and that $s_\ast \delta U = \mathcal N$. This will confirm that $U : (T\mathfrak g)_{\mathcal N} \to A$ is a Lie algebroid isomorphism. 

A straightforward computation that we leave to the reader shows that
\[
\overrightarrow U \big(\overrightarrow{\xi}{}_g, a^\uparrow_x\big) = \left( \big(\overrightarrow{ \mathcal L_g a -\xi \triangleright \mathcal L_g x}\big){}_g, 0_x \right)
\]
and
\[
\overleftarrow U \big(\overrightarrow{\xi}{}_g, a^\uparrow_x\big) = \left( - \big(\overrightarrow{\operatorname{ad}_g a}\big){}_g, - (a \triangleright x)_x^\uparrow\right),
\]
hence
\[
\delta U \big(\overrightarrow{\xi}{}_g, a^\uparrow_x\big) = \left( \big(\overrightarrow{\mathcal L_g a - \operatorname{ad}_g a -  \xi \triangleright \mathcal L_g x}\big){}_g, - (a \triangleright x)_x^\uparrow \right)
\]
for all $(g, x) \in G \ltimes \mathfrak a$, all $\xi \in \mathfrak a_{\mathrm{Lie}}$ and all $a \in \mathfrak a$, where $\overrightarrow{\xi} \in \mathfrak X (G)$ is the right invariant vector field corresponding to $\xi$. One can show that $\overrightarrow U$ is actually a Nijenhuis operator by computing da $A$-torsion of $U$, and then using Remark \ref{rem:A-torsion}. This is easy: for every $a \in \mathfrak a$, denote by $c_a \in \Gamma (A)$ the constant section equal to $a$. We have $[c_a, c_b]_A = c_{[a,b]_{\triangleright}}$ and $\rho_A (c_a) = X_{L(a)}$ for all $a,b \in \mathfrak a$. Additionally, $U a^\uparrow = c_a$. Hence
\[ 
\begin{aligned}
T^A_U (a^\uparrow, b^\uparrow) & = [Ua^\uparrow, Ub^\uparrow]_A - U[\rho_A U a^\uparrow, b^\uparrow] - U [a^\uparrow, \rho_A U b^\uparrow] \\
& = [c_a, c_b]_A - U[\rho_A c_a, b^\uparrow] - U[a^\uparrow, \rho_A c_b] \\
& = c_{[a,b]_\triangleright} - U[X_{L(a)}, b^\uparrow] - U[a^\uparrow, X_{L(b)}] \\
& = c_{[a,b]_\triangleright} - U (a \triangleright b)^\uparrow + U(b \triangleright a)^\uparrow \\
& = c_{[a,b]_\triangleright} - c_{a \, \triangleright \,b} + c_{b \, \triangleright \,a} = 0
\end{aligned}
\]
where we used that any two constant vector fields commute. By linearity we conclude that $T_U^A = 0$.

Finally, it is easy to see that
\[
s_\ast \circ \delta U \big(\overrightarrow{\xi}{}_g, a^\uparrow_x\big) = -(a \triangleright x)_x^\uparrow = \mathcal N a^\uparrow_x = \mathcal N \circ s_\ast \big(\overrightarrow{\xi}{}_g, a^\uparrow_x\big)
\]
i.e.~$\mathcal N$ is exactly the $s$-projection of $\delta U$ (equivalently $\rho_A \circ U = \mathcal N$). As a sanity check we also compute
\[
\begin{aligned}
t_\ast \circ \delta U \big(\overrightarrow{\xi}{}_g, a^\uparrow_x\big) & = \big( (\xi \triangleright \mathcal L_g x - \mathcal L_g a + \operatorname{ad}_g a) \triangleright \mathcal L_g x - \mathcal L_g (a \triangleright x) \big)^\uparrow_{\mathcal L_g x}\\
& = \big( (\xi \triangleright \mathcal L_g x - \mathcal L_g a) \triangleright \mathcal L_g x \big)^\uparrow_{\mathcal L_g x} \\
& = \mathcal N \big(\mathcal L_g a - \xi \triangleright \mathcal L_g x\big)_{\mathcal L_g x}^\uparrow \\
& = \mathcal N \circ t_\ast \big(\overrightarrow{\xi}{}_g, a^\uparrow_x\big),
\end{aligned}
\]
where we also used Lemma \ref{lem:final}. This confirms that $\mathcal N$ is also the $t$-projection of $\delta U$.
 
 \bigskip
 
\textbf{Acknowledgements.} L.V.{} is member of the GNSAGA of INdAM.

\end{document}